\documentclass[12pt, a4paper]{article}
\usepackage{graphicx} 

\usepackage{tikz}

\usepackage{amsmath,amsthm,amscd,amssymb,eucal,mathrsfs}
\usepackage{amsfonts}
\usepackage{latexsym}
\usepackage{graphicx} 
\usepackage{mathtools}
\usepackage{multicol}
\usepackage{dsfont}
\usepackage[all]{xy}
\usepackage{multirow}
\usepackage{color}
\usepackage{mathtools}
\usepackage[hidelinks]{hyperref}
\usepackage[all]{xy}
\usepackage{tcolorbox}

\newtheorem{thm}{Theorem}[section]

\newtheorem{remark}[thm]{Remark}
\newtheorem{definition}[thm]{Definition}

\newtheorem{example}[thm]{Example}

\newtheorem{alg}[thm]{Algorithm}

\newtheorem{conjecture}[thm]{Conjecture}

\newtheorem*{Satz*}{Satz}

\newtheorem{Lemma}[thm]{Lemma}

\newtheorem{proposition}[thm]{Proposition}

\newtheorem{Corollary}[thm]{Corollary}

\newcommand{\mathset}[1]{{\left\{#1\right\}}}
\newcommand{\absolute}[1]{\left\lvert#1\right\rvert}
\newcommand{\norm}[1]{\left\|#1\right\|}

\DeclareMathOperator{\Spec}{Spec}

\DeclareMathOperator{\Child}{Child}

\DeclareMathOperator{\dg}{dg}
\DeclareMathOperator{\hypergeo}{F}
\DeclareMathOperator{\leaf}{Leaf}

\title{Local ultrametric approximation of graph distance based Laplacian diffusion
}
\author{Patrick Erik Bradley
}
\date{\today}

\begin{document}

\maketitle

\begin{abstract}
The error estimation for eigenvalues and eigenvectors of a small positive symmetric
perturbation on the spectrum of a graph Laplacian is related to Gau{\ss} hypergeometric functions. Based on this, a heuristic polynomial-time algorithm for finding an optimal locally ultrametric approximation of a graph-distance power Laplacian matrix via the Vietoris-Rips graph based on the graph distance function is proposed.
In the end, the error in the solution to the graph Laplacian heat equation given by extension to a locally $p$-adic equation is estimated. 
\end{abstract}



\section{Introduction}

Complex systems, whether living or not, can under certain aspects be viewed as large networks. These in turn can be modelled as edge-weighted graphs.
In order to simulate processes on complex systems, the graph Laplacian corresponding to the underlying network needs to be processed, as it appears in many equations describing processes, like e.g.\ diffusion. An important step is the computation of eigenvalues of the graph Laplacian. In the case of a very large graph, this is prohibitively inefficient. For this reason, approximative methods for eigenvalue computations are helpful. However, if the kernel function of the graph Laplacian is a function of an ultrametric distance on the vertex set, then this computational issue is much more tractable. The discrepancy between a finite metric and a global ultrametric might be too large, as is known in cluster analysis, here it is proposed to use a more optimal local ultrametric. 
The last ingredient is the theory of $p$-adic graph Laplacians which embeds the discrete graph theory into an analytic theory on continuous spaces. This can be used as a framework in order to capture finite graphs with an indefinite number of vertices and edges, or in order to allow for graph updates. The latter leads to a $p$-adic theory of diffusion on time-dependent graphs \cite{NonAutonomous,Ledezma-Energy}.
\newline

The context of this work is in the application of ultrametric analysis to complex systems, and concretely in the project \emph{Distributed Simulation of Processes in Buildings and City Models} this theory is proposed
as a means to efficiently effect parallel computations of process simulations like heat flows on large graph datasets and includes the assessment of errors with respect to equations with actual, but large, graph Laplacians. 
The topologies of a building or city model are finite and satisfy the $T_0$-separation axiom, and such have a natural graph representation, as was observed by Alexandrov \cite{Alexandrov1937}. 
Inspired by \cite{ZunigaNetworks,Zuniga2022}, ultrametric diffusion on multiple $T_0$-topologies on a finite set is introduced and studied in \cite{IndexTopo_p}.
What would be of interest is to be able to view the theory of ultrametric diffusion on graph-represented topologies as approximating Brownian motion
on continuous ultrametric spaces via some kind of scaling limit, similarly as in 
\cite{Weisbart2024} for the $p$-adic numbers or 
\cite{PW2024} for the $p$-adic integers, and furthermore to view this in some way as an approximation to classical Brownian motion on continuous real spaces.
\newline 

The aims of the present article can be summarised as follows:
\begin{enumerate}
\item Partition a finite weighted graph in such a way that each piece as well as the quotient graph whose vertices are those pieces, are as ultrametric as possible. 

\item Estimate the error in the solution of the heat equation where the Laplacian kernel function is a power of the inverse graph distance between vertices by approximating the graph distance with its subdominant ultrametric.

\end{enumerate}

This approach relates locally hierarchic data to Mumford curves via local $p$-adic encodings of and suitable Radon measures on the pieces within an infinite compact  space, as described in \cite{IndexTopo_p,locUltra}.
\newline

The analytic approach can be summarised as follows: Given a finite metric space $(X,d)$ and its subdominant ultrametric space $(X,\delta)$, the corresponding adjacency matrices $A_d$ and $A_\delta$ satisfy:
\[
A_d=A_\delta + A_\varepsilon
\]
where 
$\varepsilon(x,y)=d(x,y)-\delta(x,y)$ for $x,y\in X$.
Then use perturbation theory as in \cite{Baumgaertel1985}
and \cite[II.1]{Kato1980} to approximate the spectrum of the corresponding Laplacian $L_d$. The two matrices $A_d$ and $A_\delta$ share a common minimum spanning tree (MST) $T$, on which 
\[
d(x,y)=\delta(x,y)
\]
for any edge $(x,y)$ of $T$.
\newline

The first aim is now placed into this context as follows:
Given an edge-weighted finite simple graph $G=(V,E,w)$, partition the vertex set $V$:
\[
V=\bigsqcup\limits_{i=1}^m V_i
\]
and obtain the following induced subgraphs:
\[
G_i=(V_i,E_i,w_i)
\]
by restriction, and the quotient graph
\[
H=(V_H,E_H,w_H)
\]
with 
\[
V_H=\mathset{V_1,\dots,V_{\absolute{V}}}
\]
and 
\[
w_H(e)=\min\mathset{w(e')\mid \text{$e'\in E(G)$ between  representatives of $o(e)$ and $t(e)$}}
\]
where $o(e)$ and $t(e)$ denote the origin and the terminus of edge $e$, respectively.
The task is now to replace each adjacency matrix of $G_i$, as well as that of $H$ with the subdominant ultrametric of the induced graph distance on $G_i$ and then to approximate the spectrum using perturbation theory. 
\newline

The deriviation of  error bounds for approximating the solutions of the heat equation for $G$ with solutions of the new ultrametric Laplacian heat equations is tackled in two ways: first, a general upper bound is proven for each coefficient in the analytic expansion of eigenvalues.
This leads to hypergeometric functions, cf.\ e.g.\ \cite{Yoshida1997}, and is a result of independent interest.
Based on this, the second approach is to give intervals containing the perturbed eigenvalues, and this comes from enclosing the graph distance in between its subdominant ultrametric and a minimal larger rescaled version thereof. This then results in an error bound given by this interval.

\section{Affine-Linear Matrix Perturbations}



Let $A(t)$ for $t\in T$ be a family of $N\times N$ matrices in $\mathds{C}^{N\times N}$.
According to \cite[Matrix Proposition (e)]{KMR2011},
if $T = \mathds{R}\ni t \mapsto A(t)$ is a so-called \emph{$C^Q$-curve} of Hermitian matrices, then the
eigenvalues and the eigenvectors can be chosen to be in the function space 
$C^Q$ which contains the real-analytic functions.
A particular instance of this is when $A(t)$ is an affine-linear map:
\begin{align}\label{alpert}
A(t)=A_0+tA_1
\end{align}
with $A_i$ symmetric real-valued $N\times N$-matrices for $i=0,1$.
Since in this case, the eigenvalues are real-analytic functions, they can in principle be evaluated at $t=1$, in order to calculate the eigenvalues of $A(1)=A_0+A_2$, or at least, in principle, to estimate the error of approximating eigenvalue $\lambda(t)$ with its constant term $\lambda_0$ which is an eigenvalue of $A_0$.
\newline

In the following, it is assumed that in (\ref{alpert}), $A_1$ is positive semi-definite, and speak of a \emph{positive perturbation}.

\subsection{General error bound}

Using the notation from \cite{Bamieh}, write
\[
\Lambda_k=\dg\left(V_0^\top A_1 V_{k-1}\right)
\]
for $k\in\mathds{N}$, where 
$\dg$ converts a square matrix to its diagonal part, $V_0$ is an orthonormal basis of eigenvectors of the  symmetric matrix $A_0$, and
\begin{align}\label{higherCompactFormula}
V_k=V_0\left(
\Pi^{\dagger 0}\circ V_0^\top
\left(
V_0\Lambda_k+V_1\Lambda_{k-1}+\dots+V_{k-1}\Lambda_1-A_1V_{k-1}
\right)
\right)
\end{align}
cf.\ \cite[\S 4]{Bamieh}.
It is assumed that the eigenvalues $\lambda(t)$  and eigenvectors $v(t)$  of $A(t)$ are analytic in $t$. Then $\Lambda_k$ and $V_k$ contain the coefficients of their analytic expansions in the diagonal and in the columns, respectively. 
Using the Frobenius norm 
$\norm{\cdot}_{F}$, obtain the following Proposition, where
$d_{\min}$ is the smallest positive entry of the adjacency matrix $A_0$.

\begin{proposition}\label{k-pert-eigenvector}
The following  holds true:
\begin{align}
\norm{V_k}_F
&\le c(k)\cdot d_{\min}^{-k}\cdot\norm{A_1}_F^k
\end{align}
where $c(k)$ is the integer sequence
\[
c(k)=c(k-1)+\sum\limits_{i=0}^{k-1}c(i)c(k-i-1)
\]
with initial conditions
\[
c(0)=c(1)=1
\]
and $k\ge1$.
\end{proposition} 

\begin{proof}
This follows by induction: for $k=0$, the inequality is satisfied, also for $k=1$, where the recursion formula starts. Assume now the inequality valid for $1\le\ell<k$. Then
\begin{align*}
\norm{V_k}_F&\le d_{\min}^{-1}
\sum\limits_{\ell=0}^{k-1}
\norm{V_\ell}_F\norm{\Lambda_{k-\ell}}_F+
d_{\min}^{-1}\norm{A_1}_F\norm{V_{k-1}}_F
\\
&\le
d_{\min}^{-1}\sum\limits_{\ell=0}^{k-1}
\norm{V_\ell}_F\norm{A_1}_F
\norm{V_{k-\ell-1}}_F+d_{\min}^{-1}\norm{A_1}_F\norm{V_{k-1}}_F
\\
&\le d_{\min}^{-1}\norm{A_1}_F\sum\limits_{\ell=0}^{k-1}
\norm{V_\ell}_F\norm{V_{k-\ell-1}}_F+d_{\min}^{-1}\norm{A_1}_F\norm{V_{k-1}}_F
\\
&\le
d_{\min}^{-1}\norm{A_1}_F\sum\limits_{\ell=0}^{k-1}c(\ell)c(k-\ell-1)d_{\min}^{1-k}\norm{A_1}_F^{k-1}
+d_{\min}^{-k}\norm{A_1}_F^kc(k-1)
\\
&= d_{\min}^{-k}\norm{A}_F^kc(k)
\end{align*}
where the last inequality uses the induction hypothesis. This proves the assertion.
\end{proof}

\begin{Corollary}
In  the case of a positive perturbation (\ref{alpert}), the following holds true:
\[
\norm{\Lambda_{k+1}}_F\le
c(k)\cdot d_{\min}^{-k}\cdot\norm{A_1}_F^{k+1}
\]
for $k\ge1$.
\end{Corollary}

\begin{proof}
This follows from 
\[
\norm{\Lambda_{k+1}}_F\le\norm{V_0^\top A_1 V_{k}}_F\le\norm{A_1}_F\norm{V_{k}}_F
\]
and Proposition \ref{k-pert-eigenvector}.
\end{proof}

\begin{Lemma}\label{coefficient}
It holds true that
\[
c(m)=\frac{(-1)^m\cdot  6^m}{2}\cdot\frac{\sqrt{\pi}\cdot
\sideset{_2}{_1}{\hypergeo}\left(
\frac{1-m}{2},-\frac{m}{2};\frac32-m;\frac19
\right)}{2\Gamma(\frac32-m)\Gamma(m+1)}
\]
for $m\ge3$,
where $\sideset{_2}{_1}{\hypergeo}(a,b;c;z)$ is the Gaussian hypergeometric function.
\end{Lemma}

\begin{proof}
Let $C(x)$ be a generating series for $c(k)$ like this:
\[
C(x)=\sum\limits_{k=0}^\infty
c(k)x^{k+1}
\]
It safisfies the equation
\[
C(x)^2+(x-1)C(x)+x=0
\]
whose solutions are
\[
C(x)=\frac{x-1}{2}\pm\frac12
\sqrt{(x-1)^2-4x}
=
\frac{x-1}{2}\pm\frac12
\left(x^2-6x+1\right)^{\frac12}
\]
of which the latter term can be expanded as
\begin{align*}
(x^2-6x+1)^{\frac12}&=\sum\limits_{n=0}^\infty
{\frac12\choose n}(x^2-6x)^n
\\
&=\sum\limits_{n=0}^\infty\sum\limits_{k=0}^n
{\frac12\choose n}{n\choose k}(-6)^{n-k}x^{k+n}
\\
&=\sum\limits_{m=0}^\infty
\sum\limits_{k=0}^m{\frac12\choose m-k}{m-k\choose k}(-6)^{m-2k}x^m
\\
&=\sum\limits_{m=0}^\infty
(-1)^m\cdot 6^m\cdot b(m)\, x^m
\end{align*}
where
\[
b(m)=\sum\limits_{k=0}^m{\frac12\choose m-k}{m-k\choose k}\cdot 6^{-2k}
\]
for $m\in\mathds{N}$.
Now, using Lemma \ref{hypergeometricCoefficient} below
 proves the assertion by inspecting the coefficients of the $C(x)$.
 \end{proof}

\begin{Lemma}\label{hypergeometricCoefficient}
It holds true that
\begin{align*}
b(m)
&={\frac12 \choose m}
\sideset{_2}{_1}{\hypergeo}\left(
\frac12-\frac{m}{2},-\frac{m}{2};\frac32-m;\frac19
\right)
\\
&=
\frac{\sqrt{\pi}\cdot
\sideset{_2}{_1}{\hypergeo}\left(
\frac12-\frac{m}{2},-\frac{m}{2};\frac32-m;\frac19
\right)}{2\Gamma(\frac32-m)\Gamma(m+1)}
\end{align*}
for $m\in\mathds{N}$.
\end{Lemma}

\begin{proof}
It holds true that
\begin{align*}
\sideset{_2}{_1}{\hypergeo}\left(
\frac12-\frac{m}{2},-\frac{m}{2};\frac32-m;z
\right)=\sum\limits_{k=0}^\infty
\frac{\left(\frac32-\frac{m}{2}\right)_k\left(-\frac{m}{2}\right)_k}{\left(\frac32-m\right)_k}
\frac{z^k}{k!}
\end{align*}
and since
\[
\left(\frac32-\frac{m}{2}\right)_k=0
\]
for $m=2k+1$, and
\[
\left(-\frac{m}{2}\right)_k=0
\]
for $m=2k-2$, whereas
\[
\left(\frac32-m\right)_k\neq0
\]
for any $k,m\in\mathds{N}$, it holds  true that
$\sideset{_2}{_1}{\hypergeo}\left(
\frac12-\frac{m}{2},-\frac{m}{2};\frac32-m;z
\right)$ is in fact a polynomial in $z$ with rational coefficients, and of degree no larger than $\frac{m}{2}$.

\smallskip
Now, observe first that
\[
{\frac12\choose m}=\frac{\Gamma\left(\frac12\right)}{\Gamma(m+1)\Gamma\left(\frac12-m+1\right)}=\frac{\sqrt{\pi}}{\Gamma(m+1)\Gamma\left(\frac12-m+1\right)}
\]
since
\[
\Gamma\left(\frac12\right)=\sqrt{\pi}
\]
holds true. This means that
it suffices to show that
\[
A:={\frac12\choose m}\frac{\left(\frac32-\frac{m}{2}\right)_k\left(-\frac{m}{2}\right)_k}{\left(\frac32-m\right)_k k!}
={\frac12\choose m-k}{m-k \choose k} \cdot 4^{-k}
\]
for $k=0,\dots,\lfloor m/2\rfloor-1$.
For this,
observe that right hand side equals
\begin{align*}
B&:=\frac{\Gamma(\frac12)\cdot 2^{-2k}}{\Gamma\left(\frac32-m+k\right)\Gamma(k+1)\Gamma(m+1-2k)}
\cdot
\frac{\Gamma(m+1)\Gamma\left(\frac32-m\right)}{\Gamma(m+1)\Gamma\left(\frac32-m\right)}
\\
&={\frac12\choose m}\cdot\frac{1}{k!}
\cdot\frac{\Gamma\left(\frac32-m\right)}{\Gamma\left(\frac32-m+k\right)}
\frac{\Gamma(m+1)}{2^{2k}\cdot\Gamma(m+1-2k)}
\end{align*}
Now, according to Legendre's duplication formula \cite[Ch.\ 6.1.18]{AS1972}, it holds true that
\[
\Gamma(m+1-2\ell)=\Gamma\left(\frac12+\frac{m}{2}-\ell\right)
\Gamma\left(1+\frac{m}{2}-\ell\right)
\cdot(2\pi)^{-\frac12}\cdot 2^{m-2k+\frac12}
\]
for $\ell=0,\dots,\lfloor m/2\rfloor-1$. And by Euler's reflection formula \cite[Ch.\ 6.1.17]{AS1972}, it holds true that
\begin{align*}
\Gamma\left(\frac12+\frac{m}{2}-\ell\right)\Gamma\left(\frac12-\frac{m}{2}+\ell\right)
&=\frac{\pi}{\sin\left(\frac{1+m}{2}\cdot\pi\right)}
\\
\Gamma\left(1+\frac{m}{2}-k\right)\Gamma\left(-\frac{m}{2}+\ell\right)
&=\frac{\pi}{\sin\left(\frac{m}{2}\cdot\pi\right)}
\end{align*}
for the same range of $\ell$.
Taking $\ell=0$ and $\ell=k$, this now yields
\begin{align*}
B&={\frac12\choose m}\cdot\frac{1}{k!}
\cdot\frac{\Gamma\left(\frac12-\frac{m}{2}+k\right)}{\Gamma\left(\frac12-\frac{m}{2}\right)}
\cdot\frac{\Gamma\left(-\frac{m}{2}+k\right)}{\Gamma\left(-\frac{m}{2}\right)}\cdot\frac{\Gamma\left(\frac32-m\right)}{\Gamma\left(\frac32-m+k\right)}=A
\end{align*}
as asserted.
\end{proof}

Further properties of the Gamma function can be found in \cite[Ch.\ 6.1]{AS1972}.

\begin{Lemma}\label{GammaStatements}
The following statements hold true:
\begin{align}
\lim\limits_{m\to\infty}
\left[\Gamma\left(
\frac32-m\right)^{-1}
\cdot\Gamma(m+1)^{-1}\right]
&=0
\\[2mm]
\lim\limits_{m\to\infty}\left[\Gamma\left(1-\frac{m}{2}\right)^{-1}
\cdot\Gamma\left(
\frac32-\frac{m}{2}
\right)^{-1}
\cdot\Gamma(m+1)^{-1}\right]
&=0
\end{align}
\end{Lemma}

\begin{proof}
It holds true that
\begin{align*}
\Gamma\left(
\frac32-m
\right)^{-1}
&=\Gamma\left(
1-m+\frac12
\right)^{-1}
\\
&=\Gamma(1-m)\cdot
\pi^{-\frac12}\cdot 2^{-1+2(1-m)}\cdot
\Gamma(2-2m)^{-1}
\end{align*}
where the last equality is Legendre duplication \cite[Ch.\ 6.1.18]{AS1972}. 
This now implies the first statement, since
\[
\lim\limits_{m\to\infty}\frac{\Gamma(1-m)}{\Gamma(2-2m)}=0
\]
holds true.

\smallskip
As for the second assertion, observe that 
\begin{align*}
\Gamma\left(\frac32-\frac{m}{2}\right)&=\Gamma\left(1-\frac{m}{2}+\frac12\right)
\\
&=\Gamma\left(1-\frac{m}{2}\right)^{-1}\cdot\sqrt{\pi}\cdot 2^{1-2(1-m/2)}\cdot \Gamma(2-m)  
\end{align*}
where the last equality is again Legendre duplication \cite[Ch.\ 6.1.18]{AS1972}. Now,
\[
\Gamma(2-m)=\Gamma(1-(m-1))
=\Gamma(m-1)^{-1}\frac{\pi}{\sin\left(\frac{m-1}{2}\pi\right)}
\]
which implies that
\[
\Gamma\left(\frac32-\frac{m}{2}\right)
=\frac{\pi^{\frac32}\cdot 2^{m-1}}{\Gamma\left(1-\frac{m}{2}\right)\Gamma(m-1)\sin\left(\frac{m-1}{2}\pi\right)}
\]
from which the second statement follows.
\end{proof}

\begin{Corollary}\label{zeroSeq}
It holds true that
$\lim\limits_{m\to\infty}\absolute{b(m)}=0$.
\end{Corollary}

\begin{proof}
Up to a positive constant factor, $b(m)$ is the function
\[
\gamma(m)=
\frac{\sideset{_2}{_1}{\hypergeo}\left(\frac12-\frac{m}{2},-\frac{m}{2},\frac32-m;\frac19\right)}{\Gamma\left(\frac32-m\right)\Gamma(m+1)}
\]
Now,
\[
\sideset{_2}{_1}{\hypergeo}\left(
\frac12-\frac{m}{2},-\frac{m}{2},\frac32-m;\frac19
\right)<\sideset{_2}{_1}{\hypergeo}\left(\frac12-\frac{m}{2},-\frac{m}{2},\frac32-m;1\right)
\]
The Gauss summation theorem says that
\[
\sideset{_2}{_1}{\hypergeo}(\alpha,\beta,\gamma;1)
=\frac{\Gamma( \gamma)\Gamma(\gamma-\alpha-\beta)}{\Gamma(\gamma-\alpha)\Gamma(\gamma-\beta)}
\]
holds true, if $\Re(\gamma-\alpha-\beta)>0$ and $\gamma$ is not a negative integer \cite[\S 1.3]{Bailey1935}. Since this
is the case for the entries above, it now follows that
\[
\sideset{_2}{_1}{\hypergeo}\left(
\frac12-\frac{m}{2},-\frac{m}{2},\frac32-m;1
\right)
=\frac{\Gamma\left(\frac32-m\right)\Gamma(1)}{\Gamma\left(1-\frac{m}{2}\right)\Gamma\left(\frac32-\frac{m}{2}\right)}
\]
for $m>1$. Applying Lemma \ref{GammaStatements} now yields the assertion.
\end{proof}

\begin{thm}\label{errorBound1}
Assume (\ref{alpert}).
Then it holds true that
\begin{align*}
\absolute{\lambda(1)-\lambda_0}
&\le C_\lambda\cdot \norm{A_1}_F
\cdot
\left(1+
\frac{\norm{A_1}_F}{\frac{d_{\min}\rule[-1mm]{0pt}{3mm}}{6}}
\right)^{-1}
\\
\norm{v(1)-v_0}_2
&\le
C_v\cdot\norm{A_1}_F
\cdot
\left(
1+\frac{\norm{A_1}_F}{\frac{d_{\min}\rule[-1mm]{0pt}{3mm}}{6}}
\right)^{-1}
\end{align*}
for some $C_\lambda,C_v>0$,
if $\norm{A_1}_F<\frac{d_{\min}\rule[-1mm]{0pt}{3mm}}{6}$. 
\end{thm}

\begin{proof}
According to Lemma \ref{coefficient}, it holds true that
\[
c(k)=\frac12\cdot(-6)^k\cdot b(k)
\]
for $k\ge3$, and with $b(k)\to 0$ according to Corollary \ref{zeroSeq}.  Hence, using Proposition \ref{k-pert-eigenvector}, it now follows that
\begin{align*}
\absolute{\lambda(1)-\lambda_0}
&\le C\norm{A_1}_F
\sum\limits_{k=0}^\infty
\frac{(-6)^m \norm{A_1}_F^k}{d_{\min}^k}
\\
&=C\cdot\norm{A_1}_F\cdot
\left(1+\frac{\norm{A_1}_F}{\frac{d_{\min}\rule[-1mm]{0pt}{3mm}}{6}
}
\right)^{-1}
\end{align*}
for some $C>0$, if $\norm{A_1}_F<\frac{d_{\min}\rule[-1mm]{0pt}{3mm}}{6}$,
 as asserted. The second equality thus  also follows from Proposition \ref{k-pert-eigenvector}.
\end{proof}

The application scenario of this article is to locally replace the graph distance with the subdominant ultrametric. Using the Vietoris-Rips graph from topological analysis, in which clusters are obtained as connected components by eliminating edges of length above a given threshold, 
the following now seems plausible, where
\[
L_0^{(\alpha)}(v,w)=\delta_\epsilon(v,w)^{-\alpha}
\]
is the kernel function to be used for a Laplacian operator on the finite graph, where $\delta_\epsilon$ is the subdominant ultrametric as described above.
The actual Laplacian on a finite weighted graph $G$ used here, is
given by
\[
L(v,w)=d_G(v,w)^{-\alpha}
\]
as its kernel function for $\alpha>0$.

\begin{conjecture}
The Vietoris-Rips clusters for some $\epsilon>0$ do realise the sufficient ultrametric condition of
\[
\norm{L_1^{(\alpha)}}_F<\frac{d_{\min}}{6}
\]
from Theorem \ref{errorBound1}
for each cluster as well as for the inter-cluster graph, and some $\alpha>0$.
\end{conjecture}

\begin{definition}
A \emph{minimising} Vietoris-Rips graph $\Gamma_\epsilon$ is one such that
\[
M(\epsilon)=\sum\limits_{C\in\mathcal{C}(\Gamma_\epsilon)}
\frac{\norm{L_{1,C}^{(\alpha)}}_F}{\frac{d_{\min}^C\rule[-2mm]{0pt}{3mm}}{6}}
+\frac{\norm{L_{1,H}^{(\alpha)}}_F}{\frac{d_{\min}^H\rule[-2mm]{0pt}{3mm}}{6}}
\]
is minimal, where $d^{C}_{\min}$ is the minimal non-trivial spectral distance for the ultrametric Laplacian operator on cluster $C$, and $d_{\min}^H$ likewise on the quotient graph $H$, defined as the inter-cluster graph for $\Gamma_\epsilon$, whose edge weights are the inter-cluster distances. 
\end{definition}

\subsection{Error estimation via the subdominant ultrametric}

The Courant-Fisher Theorem implies that perturbing a Hermitian matrix $A$ with a positive semidefinite matrix $B$ yields the inequality
\[
\lambda_{A,k}\le\lambda_{A+B,k}
\]
for the $k$-th eigenvalue (in their natural ordering as real numbers). This leads to the following error intervals, given the
 Laplacian matrices:

\begin{align*}
({L}_0f(v))&=\sum\limits_{w\in V}d_G(v,w)^{-\alpha}(f(v)-f(w))
\\
({L}f)(v)&=\sum\limits_{w\in V}\delta(v,w)^{-\alpha}(f(v)-f(w))
\\
({L}'f)(v)&=\sum\limits_{w\in V}\tau^{-\alpha}({L}f)(v)
\end{align*}
where 
$d_G$ is the graph distance of the connected weighted graph $G=(V,E,w)$, $\delta$ the corresponding subdominant ultrametric, and
\[
\tau=\max\mathset{\frac{d_G(v,w)}{\delta(v,w)}\mid v,w\in V,\;v\neq w}\ge 1
\]
the \emph{dominant augmentation factor}, where it is assumed that $G$ has at least two vertices.
This implies that the eigenvalue $\lambda(1)$ of ${L}$, viewed as coming from the perturbation
\[
{L}(t)={L}_0+t{L}_1
\]
with $t=1$, satisfies the property
\begin{align}\label{interval}
\lambda(0)\in\left[\tau^{-\alpha}\lambda(1),\lambda(1)\right]
\end{align}
where $\lambda(0)=\lambda_0\in\Spec({L}_0)$.
The following section deals with estimating this interval in \emph{locally} ultrametric approximations of the Laplacian $L_0$.

\section{Locally ultrametric approximations}

In the following, locally ultrametric approximations of operator $L$ are found using the Vietoris-Rips graph of a finite metric space, and an algorithm for minimising the interval $\left[\tau^{-\alpha}\lambda(1),\lambda(1)\right]$ is given, which yields an error estimate for the perturbation of eigenvalue $\lambda_0$ by $\lambda(1)$. 

\subsection{Vietoris-Rips graph partitioning}

Take the local ultrametric $\delta_\epsilon$ of \cite[\S 2.1]{locUltra} to define an operator $H_\epsilon$
\[
(H_\epsilon f)(v)=\sum\limits_{w\in V}\delta_{\epsilon}(v,w)^{-\alpha}(f(v)-f(w))
\]
where $\delta_\epsilon$ restricted to point pairs in a connected component $C$ (called \emph{cluster}) of the Vietoris-Rips graph $\Gamma_\epsilon$ is the subdominant ultrametric for the restricted graph distance on $C$.
An optimal value of $\epsilon>0$ is defined as given by minimising the spectral distance between ${H}_\epsilon$ and ${L}_0$. The operator $H_\epsilon$ is a \emph{hierarchical Parisi matrix}, similar to the operator defined in \cite[Def.\ 3.1]{IndexTopo_p}.
\newline

A heuristic algorithm for finding an optimal $\epsilon$ is proposed as follows:

\begin{enumerate}
\item For each $\epsilon$-cluster $C$ having at least two vertices, let $\tau_C$ the dominant augmentation factor for $\delta_\epsilon$ as the subdominant ultrametric for the restriction of the graph distance
 to $C\times C$. Otherwise, let $\tau_C=0$.
\item If $\Gamma_\epsilon$ has at least two clusters, let $\tau_{H_\epsilon}$ be the dominant augmentation factor of the subdominant ultrametric for the  metric given as the inter-cluster distances, viewed as a metric on the quotient graph $H_\epsilon$, where each cluster is  collapsed to a vertex. Otherwise, let $\tau_{H_\epsilon}=0$. 
\item Minimise the function
\[
\Phi(\epsilon)=\absolute{\tau_{H_\epsilon}-1}^2+\sum\limits_{C\in\mathcal{C}(\Gamma_\epsilon)}\absolute{\tau_C-1}^2
\]
where $\mathcal{C}(\Gamma_\epsilon)$ is the set of connected components of $\Gamma_\epsilon$, 
and
$\epsilon$ ranges from $0$ to the  maximal distance in $G$.
\end{enumerate}

The minimisers of $\Phi(\epsilon)$ minimise the interval
\[
[\tau^{-\alpha}\lambda(1),\lambda(1)]
\]
for eigenvalue $\lambda(t)$ at value $t=1$ among the possible Vietoris-Rips graphs.

\begin{remark}
In the case of $d$ being an ultrametric, any edge covering of $G$ producing at least two connected components leads to a minimiser of $\Phi(\epsilon)$.
The smallest minimising $\epsilon$ is thus given by a minimal weighted edge covering of  $G$.
\end{remark}

\begin{example}
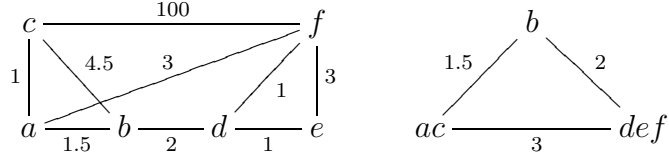
\begin{figure}[h]
\[
\xymatrix{
c\ar@{-}[rrr]^{100}\ar@{-}[dr]^{4.5}&&&f\ar@{-}[d]^3
\\
a\ar@{-}[u]^1\ar@{-}[r]_{1.5}\ar@{-}[urrr]^3&b\ar@{-}[r]_2&d\ar@{-}[r]_{1}\ar@{-}[ur]_{1}&e
}
\qquad
\xymatrix{
&b\ar@{-}[dr]^2
\\
ac\ar@{-}[rr]_{3}\ar@{-}[ur]^{1.5}&&def
}
\]
\caption{A weighted graph, and the quotient graph $H_1$.}\label{graph}
\end{figure}
Given the graph as in Figure \ref{graph} (left), 
the graph $H_1$ is 
depicted in Figure \ref{graph} (right).
Some values of $\Phi(\epsilon)$ are as follows:
\begin{align*}
\Phi(1)&=0^2+1^2+1^2+(0.5)^2=2\frac14
\\
\Phi(1.5)&=29^2+2^2+0^2=845
\\
\Phi(2)&=\frac{100^2}{4.5^2}-1+1=
493.8
\end{align*}
The minimal weight edge covering of this graph is given for $\epsilon=1.5$. However, it is not a  minimsiser for $\Phi(\epsilon)$, since $\Phi(1)$ is smaller. In fact, $\Phi(1.5)$ is even larger than $\Phi(2)$, whereby $\Gamma_2$ is the minimal spanning tree of the graph, i.e.\ $H_2$ is a singleton graph.
\end{example}

The first Betti number $g$ of the quotient graph $H$ for which $\norm{\tau(\epsilon)}_1$ is minimal, is the \emph{optimal augmentation genus}, and the one for which the spectral distance is minimal, is the \emph{optimal Mumford curve genus} for the graph $G=(V,E,w)$.
This name is meant to remind of the approach in 
\cite{IndexTopo_p,locUltra}, where after a $p$-adic embedding, the graph $H$ can be viewed as the skeleton of a Mumford curve.

\begin{remark}
It is an open question how far off an optimiser for $\tau(\epsilon)$ is from the minimiser of the spectral distance between $H_\epsilon$ and $L_0$. Similarly, the question is how different the optimal augmentation genus is from the optimal Mumford curve genus.
\end{remark}

\begin{alg}\label{algorithm}
Proceed along the following steps:
\begin{enumerate}
\item Output set $S=\emptyset$.
\item Let $F_0$ be the MST of $\mathcal{G}$ which is the MST of $G$, and
\[
\epsilon_0=\text{maximal edge length of $F_0$}
\]
Include $\epsilon_0$ into $S$.
\item Sort the edges of $F_0$ according to their lengths in descending order.
\item In Step $k+1$ for $k\in\mathds{N}$, remove from forest $F_k$ all edges with length $\epsilon_k$, and obtain the forest $F_{k+1}$. Let
\[
\epsilon_{k+1}=\text{maximal edge length in $F_{k+1}$}
\]
If $\Phi(\epsilon_{k+1})<\Phi(\epsilon_k)$, then replace all elements of $S$ with $\epsilon_{k+1}$. If $\Phi(\epsilon_{k+1})=\Phi(\epsilon_k)$, then include $\epsilon_{k+1}$ into $S$. If $\Phi(\epsilon_{k+1})>\Phi(\epsilon_k)$, then do not change $S$.
\item Output: the set $S$.
\end{enumerate}
\end{alg}

\begin{thm}
Algorithm \ref{algorithm} terminates and produces the set of minimisers of $\Phi(\epsilon)$ in polynomial time w.r.t.\ the number of vertices of $G$.
\end{thm}

\begin{proof}
Termination is guaranteed, since edges from a finite forest are removed in each step. By force, the output consists of the set of $\epsilon$ for which $\Phi(\epsilon)$ is minimal. Since the MST can be found in polynomial time, and all  $\absolute{V}-1$ edges of it are sorted, the time-complexity is polynomial in $\absolute{V}$.  
\end{proof}

\begin{remark}
The function $\Phi(\epsilon)$ punishes singleton clusters by letting each of these contribute to a summand of $\Phi(\epsilon)$. In principle, singletons could be dealt with in different ways. We conjecture that for the present definition of $\Phi(\epsilon)$, that if $d_G$ is not an ultrametric, then neither $H_\epsilon$ nor all $C\in\mathcal(\Gamma_\epsilon)$ are singletons for some optimal value of $\epsilon$.
\end{remark}

\subsection{Laplacian heat equation error estimates using a $p$-adic extension}

The goal is to estimate the error of the solution to the Cauchy problem
\begin{align}\label{CP}
\frac{\partial}{\partial t}
u(t,\cdot)-(L_0u)(t,\cdot)&=0
\\\nonumber
u(0,\cdot)=u_0(\cdot)&\in \mathds{R}^N
\end{align}
when approximated by the solution of the Cauchy problem
\begin{align}\label{CPlu}
\frac{\partial}{\partial t}
u(t,\cdot)-(H_\epsilon u)(t,\cdot)&=0
\\\nonumber
u(0,\cdot)=u_0(\cdot)&\in\mathds{R}^N
\end{align}
where $u_0=u_0(\cdot)$ is a fixed map $V\to\mathds{R}$ given as an
initial condition to both equations, and $\epsilon$ is a minimiser of $\Phi(\epsilon)$ from the previous subsection.
\newline

Useful for this are ultrametric wavelets, as defined in \cite{XK2005}, and which generalise the $p$-adic wavelets first defined by S.\ Kozyrev in \cite{Kozyrev2002}. The theory of ultrametric integral operators having ultrametric wavelets as eigenfunctions is developped in \cite{XK2005}, and this generalises the $p$-adic case of \cite{Kozyrev2004}.
\newline

The ultrametric spaces used in \cite{XK2005} are boundaries of infinite trees which are locally finite, and from the point of view of this article, these are globally ultrametric spaces, whereas here, locally ultrametric spaces are used, which locally look like those ultrametric spaces of \cite{XK2005}.
Hence, we now take an ultrametric extension of each of our finite trees which are locally finite rooted trees, and use the disjoint union of their boundary spaces as our underlying space $Z$.
Applying the methods of \cite{XK2005} on those local pieces yields 
 a Radon measure $\nu$ on $Z$ which is compatible with the locally subdominant ultrametric $d_\epsilon$ in the sense that on each ultrametric patch, the volumes of balls equal to their radii w.r.t.\ $d_\epsilon$. In this way, the vertex set $V$ can be understood as a collection of balls in $Z$ via an ultrametric embedding $V\to Z$, generalising Z\'u\~niga-Galindo's way of producing $p$-adic integral operators for finite graphs in \cite{ZunigaNetworks}. This now gives a framework for viewing the theory of finite graphs as part of ultrametric analysis.
 In order to simplify matters, the locally ultrametric extension of a finite graph $G$ used here is given by a \emph{local} $p$-adic embedding $V\to Z$, where $Z$ is a $p$-adic manifold, and the vertex set of each connected component of $C$ of $\Gamma_\epsilon$ is embedded into  a disc $B_C\subset\mathds{Q}_p$, just like in the Z\'u\~{n}iga case, and such that $Z$ is the disjoint union of the sets $B_C$. And again, $\nu$ is a Radon measure on $Z$ compatible with $d_\epsilon$ in this setting. The details of obtaining $\nu$ can be found in  \cite[\S 3.3]{IndexTopo_p} and \cite[\S 2.2]{locUltra}, where it is called \emph{equity measure}.
 \newline

A simplification arises via $L_\epsilon$ which replaces also the quotient metric with its subdominant ultrametric. 
\newline

In the following, we adapt  the notation from \cite{XK2005} to our setting as follows:

\begin{thm}\label{locallyUltrametricSpectrum}
The space $L^2(Z,\nu)$ has an orthonormal basis consisting of eigenfunctions of $H_\epsilon$, which in turn consists of 
 of ultrametric wavelets $\psi_{B,j}$ 
supported in the clopen subsets of $Z$ corresponding to the connected components of $\Gamma_\epsilon$, and 
an eigenbasis of the
Laplacian matrix associated with
\[
\nu(U_{C'})(d_\epsilon(C,C')^{-\alpha})_{C,C'\in \mathcal{C}(\Gamma_\epsilon)}
\]
for the measure $\mu$ on $V$ induced by the Radon measure $\nu$ on $Z$.
The eigenvalues corresponding to the ultrametric wavelets $\psi_{B,j}$ are
\[
\lambda_{B}=
\int_{Z\setminus B}
d_\epsilon(B,y)^{-\alpha}\,d\nu(y)
+\nu(B)^{1-\alpha}
\]
for $B\subset Z$ the disc  
corresponding to a non-root vertex of a local $\delta_\epsilon$-tree $T$.
\end{thm}

This combines \cite[Thm.\ 3.6]{IndexTopo_p}, where the vertices of a graph are interpreted as $p$-adic discs, and \cite[Thm.\ 3.13]{IndexTopo_p} where a finite ultrametric tree is $p$-adically extended.

\begin{proof}
Proceed as in \cite[Thm.\ 3.6]{IndexTopo_p}, but using 
\cite[Thm.\ 10]{XK2005} locally instead of \cite[Thm.\ 3]{Kozyrev2004} globally for the ultrametric wavelets as eigenfunctions. The other eigenfunctions, as well as the orthogonal decomposition of $L^2(Z,\nu)$, are derived in a similar way as in \cite[Thm.\ 10.1]{ZunigaNetworks}.
\end{proof}



Finite trees also have corresponding wavelets on their boundary, cf.\
\cite{GNC2010}, where they are called Haar-like wavelets. In \cite[Thm.\ 3.13]{IndexTopo_p}, these were
extended to $p$-adic domains and seen as eigenfunctions of certain ultrametric integral operators defined on these.

\begin{Corollary}
The space $L^2(Z,\nu)$ has an orthonormal basis consisting of eigenfunctions of $L_\epsilon$, which in turn consist of ultrametric wavelets supported in the clopen subsets of $Z$ corresponding to the connected components of $\Gamma_\epsilon$, and an eigenbasis of the Laplacian matrix associated with
\[
\nu(U_C')(\delta_\epsilon(C,C')^{-\alpha})_{C,C'\in\mathcal{C}(\Gamma_\epsilon)}
\]
for the measure $\mu$ on $V$ induced by Radon measure $\nu$. The corresponding eigenvectors are the Haar-like wavelets from \cite{GNC2010} with 
eigenvalues
\[
\alpha_B
=\int_{\delta_\epsilon(B,w)\ge\rho}
\delta_\epsilon(B,w)^{-\alpha}
\,d\mu(w)
+\mu(B)^{1-\alpha}
\]
for $B\subset V$ a disc of radius $\rho$ w.r.t.\ the ultrametric $\delta_\epsilon$.
\end{Corollary}

\begin{proof}
Since the only new statement is the eigenvalue $\alpha_B$, it suffices to calculate it.

\smallskip
First, it needs to be shown that the Haar-like wavelets supported in an ultrametric disc w.r.t.\ the finite ultrametric $\delta_\epsilon$ are indeed eigenfunctions. For this, the proof of \cite[Thm.\ 2]{XK2005} can be used, since it shows also in the finite ultrametric case that the Haar-like wavelets are mutually orthogonal. In order to show that these eigenfunctions together with the constant function $1$ span the function space on the finite ultrametric space, observe  that for any finite rooted tree $T$, the following formula holds true:
\[
\sum\limits_{v\in V(T)\setminus\leaf(T)}(\absolute{\Child(v)}-1)=
\absolute{\leaf(T)}-1,
\]
where $\Child(v)$ is the set of child nodes of $v$, and $\leaf(T)$ the set of all leaf nodes of $T$.
This can be seen by induction over the generations of a rooted tree.

\smallskip
Secondly, the corresponding eigenvalues can now be computed following \cite[Thm.\ 10]{XK2005}, since the condition of having only infinite geodesics from the root vertex is needed only in the case of an infinite ultrametric space. 
Namely, first observe that 
\begin{align*}
\alpha_B\psi(x)&=\int_{\delta(B,w)>\rho}
\delta_\epsilon(B,w)^{-\alpha}\,d\mu(w)\,\psi(x)
\\
&+
\int_{\delta(x,y)\le\rho}
\delta_\epsilon(x,w)^{-\alpha}(\psi(x)-\psi(w))\,d\mu(w)
\end{align*}
where $x\in B$, and also the second integral is readily seen to not depend on $x\in B$, all by using the explicit form of the Haar-like wavelet $\psi(x)$ in \cite[Def.\ 3.11]{IndexTopo_p} (where it is called ultrametric wavelet),
and then the eigenvalue $\alpha_B$ can be written as
\[
\alpha_B
=\int_{\delta_\epsilon(B,w)\ge\rho}
\delta_\epsilon(B,w)^{-\alpha}
\,d\mu(w)
+\mu(B)\cdot\rho^{-\alpha}
\]
for $\alpha>0$, where the second term is calculated as in the proof of \cite[Thm.\ 10]{XK2005}. This implies the asserted eigenvalue.
\end{proof}

Write now the heat equation on the quotient graph $\Gamma_\epsilon/\sim$ as
\begin{align}\label{CP_A}
\frac{\partial}{\partial t}u_{\overline{A}(\tau)}(x,t)=\overline{A}(
\tau
)u_{\overline{A}(\tau)}(x,t)
\end{align}
with initial condition $u_{\bar{A}}(x,0)=u_0(x)$, using the decomposition
\[
\overline{A}_\epsilon(\tau)=\overline{L}_\epsilon+\overline{B}_{\epsilon}\tau
\]
with
\[
\overline{A}(1)=\overline{H}_\epsilon
\]
the ultrametric Laplacian on the cluster graph $\Gamma_\epsilon/\!\!\sim$. Notice that $\tau$ and $t$ are meant to be two independent parameters.
Now, expand the initial condition $u_0$ as
\[
u_0=\sum\limits_{\lambda(\tau)\in\Spec A(\tau)}
\alpha_\lambda(\tau) v_\lambda(\tau)
\]
\newline

Let $u_{L_0}$, $u_{H_\epsilon}$  be solutions to (\ref{CP}), (\ref{CPlu}),  respectively having the same initial condition $u_0\in\mathds{C}^N$.
The variable in the corresponding heat equations for the quotient graph $\Gamma_\epsilon/\!\!\sim$
is written as $u_{\overline{H}_\epsilon}$. The equation for $\overline{L}_\epsilon$ is assumed to use the variable $u_{\overline{L}_\epsilon}$.
Let
\begin{align*}
\Psi_A(\bullet,\epsilon)
&=C_\bullet\norm{\overline{B}_\epsilon}_F
\left(1+\frac{\norm{\overline{B}_\epsilon}_F}{\frac{d_{\min}}{6}}
\right)^{-1}
\\
\Psi_{H_\epsilon}(\bullet,\epsilon)&
=C_\bullet\norm{A_\epsilon}_F
\left(1+\frac{\norm{A_\epsilon}_F}{\frac{d_{\min}}{6}}\right)^{-1}
\end{align*}
where in the first case $\bullet\in\mathset{v,\lambda}$,
and in the second case $\bullet\in\mathset{w,\gamma}$.
These elements stand for eigenvectors $v(\tau)$  of $\overline{A}_\epsilon(\tau)$ with eigenvalue $\lambda(\tau)$ in a fixed onb of $\mathds{C}^{b_0(\Gamma_\epsilon)}$, and 
eigenvectors
$w(\tau)$ of $H_\epsilon(\tau)$ with eigenvalue $\gamma(\tau)$ in a fixed onb of $\mathds{C}^N$.

\begin{thm}
Assume that an initial condition for (\ref{CP_A}) is given as
\[
\sum\limits_{\lambda(\tau)\in\Spec A(\tau)}\alpha_{\lambda(\tau)}v_{\lambda(\tau)},\quad
\sum\limits_{\gamma(\tau)\in\Spec H_\epsilon(\tau)}
\alpha_{\gamma(\tau)}w_{\gamma(\tau)}
\]
with
\[
a_A=(\alpha_{\lambda(1)})
\in\mathds{C}^{b_0(\Gamma_\epsilon)},
\quad
a_{H}=(\alpha_{\gamma(1)})\in\mathds{C}^N\,.
\]
Then it holds true that
\begin{align*}
\norm{u_{\overline{L}_\epsilon}-u_{\overline{H}_\epsilon}}&
\le\left(
e^{-\lambda(1)t}\Psi_A(v,\epsilon)
+2e^{-\Psi_A(\lambda,\epsilon)t}
\right)\norm{a_A}_1
\\
\norm{u_{L_0}-u_{H_\epsilon}}
&\le
\left(
e^{-\gamma(1)t}\Psi_{H}(w,\epsilon)
+2e^{-\Psi_H(\gamma,\epsilon)t}
\right)\norm{a_H}_1
\end{align*}
for 
and $t\ge0$.
\end{thm}

\begin{proof}
Write
\[
\alpha_{\lambda(1)}=\beta_{\lambda_0}+\beta_\lambda
\]
and first observe that
\begin{align*}
\norm{\alpha_{\lambda(1)} e^{-\lambda(1)t} v(1)-\beta_{\lambda_0} e^{-\lambda_0t} v_0}_2
&\le
e^{-\lambda(1)t}\absolute{\alpha_{\lambda(1)}} \norm{v(1)-v_0}_2
\\
&+
e^{-\lambda_0t}\absolute{e^{-(\lambda(1)-\lambda_0)t}\alpha_{\lambda(1)}-\beta_{\lambda_0}}\norm{v_0}_2
\end{align*}
Now,
\begin{align*}
\absolute{
e^{-(\lambda(1)-\lambda_0)t}\alpha_{\lambda(1)}
-\beta_{\lambda_0}
}
&\le\absolute{
e^{-(\lambda(1)-\lambda_0)t}\alpha_{\lambda(1)}-\alpha_{\lambda(1)}
}+\absolute{\alpha_{\lambda(1)}-\beta_{\lambda_0}}
\\
&\le\absolute{\alpha_{\lambda(1)}}+\absolute{\alpha_{\lambda(1)}-\beta_{\lambda_0}}\,.
\end{align*}
It follows that
\begin{align*}
\norm{u_{\overline{L}_\epsilon}-u_{\overline{H}_\epsilon}}_2
&\le
\sum\limits_{\lambda(1)\in\Spec\overline{A}(1)}
\norm{
\alpha_{\lambda(1)}e^{-\lambda(1)t}
v(1)-\beta_{\lambda_0}e^{-\lambda_0t}v_0
}_2
\\
&\le
e^{-\lambda(1)t}\norm{v(1)-v_0}_2
\sum\limits_{\lambda(1)}\absolute{\alpha_{\lambda(1)}}
+e^{-\lambda_0t}\norm{v_0}_2
\sum\limits_{\lambda(1)}\absolute{\alpha_{\lambda(1)}-\beta_{\lambda_0}}
\\
&\le
e^{-\lambda(1)t}\norm{v(1)-v_0}_2
\sum\limits_{\lambda(1)}\absolute{\alpha_{\lambda(1)}}
+e^{-\lambda_0t}
\sum\limits_{\lambda(1)}\absolute{\alpha_{\lambda(1)}-\beta_{\lambda_0}}\,.
\end{align*}
Write
\[
b_A=(\beta_{\lambda_0})\in\mathds{C}^{b_0(\Gamma_\epsilon)}
\]
Then
\[
b_A=Oa_A
\]
for some unitary matrix $O$, and
\[
\norm{a_A-b_A}_1
=\norm{a_A-Oa_A}_1\le\norm{1-O}\norm{a_A}_1\le 2\norm{a_A}_1
\]
which yields
\[
\norm{u_{\overline{L}_\epsilon}-u_{\overline{H}_\epsilon}}
\le 
\left(e^{-\lambda(1)t}\norm{v(1)-v_0}_2+2e^{-\lambda_0 t}
\right)\norm{a_A}_1\,.
\]
Since
\[
\lambda_0\ge\lambda(1)-
\underbrace{C_{\lambda}
\norm{\overline{B}_\epsilon}_F
\cdot\left(1+\frac{\norm{\overline{B}_\epsilon}_1}{d_{\min}/6}\right)^{-1}}_{=\Psi_{{A}}(\lambda,\epsilon)}\,,
\]
this yields
\[
\norm{u_{\overline{L}_\epsilon}-u_{\overline{H}_\epsilon}}
\le\left(
e^{-\lambda(1)t}\Psi_{{A}}(v,\epsilon)
+2e^{-\Psi_A(\lambda,\epsilon)t}
\right)\norm{a_A}_1
\]
For the same reason,
\[
H_\epsilon(\tau)=L_0+\tau A_\epsilon
\]
yields
\[
\norm{u_{L_0}-u_{H_\epsilon}}
\le \left(
e^{-\gamma(t)t}\Psi_{H}(w,\epsilon)+2e^{-\Psi_H(\gamma,\epsilon)t}
\right)\norm{a_H}_1
\]
This proves the assertion.
\end{proof}







\section*{Acknowledgements}

\'Angel Mor\'an Ledezma,  
Paulina Halwas, Leon Nitsche,  
David Weisbart and Wilson Z\'u\~niga-Galindo are warmly thanked for fruitful discussions. 
This work is partially supported by the Deutsche Forschungsgemeinschaft under project number 469999674.

\bibliographystyle{plain}
\bibliography{biblio}

\begin{thebibliography}{10}

\bibitem{AS1972}
M.~Abramowitz and I.A. Stegun.
\newblock {\em Handbook of Mathematical Functions with Formulas, Graphs and
  Mathematical Tables}.
\newblock Applied Mathematics Series 55. National Bureau of Standards, 1972.
\newblock 10th Printing.

\bibitem{Alexandrov1937}
P.S. Alexandrov.
\newblock Diskrete {R\"aume}.
\newblock {\em Matematicheskii Sbornik (N.S.)}, 2:501--518, 1937.

\bibitem{Bailey1935}
W.N. Bailey.
\newblock {\em Generalized hypergeometric series}.
\newblock Cambridge tracts in mathematics and mathematical physics 32.
  Cambridge University Press, 1935.

\bibitem{Bamieh}
B.~Bamieh.
\newblock A tutorial on matrix perturbation theory (using compact matrix
  notation).
\newblock arXiv:2002.05001 [math.SP], 2020.

\bibitem{Baumgaertel1985}
H.~Baumgärtel.
\newblock {\em Analytic Perturbation Theory for Matrices and Operators},
  volume~15 of {\em Operator Theory}.
\newblock Birkh\"auser, Basel, 1985.

\bibitem{locUltra}
P.E. Bradley.
\newblock On the local ultrametricity of finite metric data.
\newblock arXiv:2408.07174 [cs.IR], 2024.

\bibitem{IndexTopo_p}
P.E. Bradley and \'A.M. Ledezma.
\newblock Approximating diffusion on finite multi-topology systems using
  ultrametrics.
\newblock arXiv:2411.00806 [cs.DM], 2024.

\bibitem{NonAutonomous}
P.E. Bradley and \'A.M. Ledezma.
\newblock A non-autonomous p-adic diffusion equation on time changing graphs.
\newblock arXiv:2407.21555 [math.AP], To appear in Reports on Mathematical
  Physics.

\bibitem{GNC2010}
M.~Gavish, B.~Nadler, and R.R. Coifman.
\newblock Multiscale wavelets on trees, graphs and high dimensional data:
  Theory and applications to semi supervised learning.
\newblock In {\em Proceedings of the 27th International Conference on Machine
  Learning}, Haifa, Israel, 2010.

\bibitem{Kato1980}
T.~Kato.
\newblock {\em Perturbation Theory for Linear Operators}.
\newblock Classics in Mathematics. Springer, Berlin Heidelberg, reprint of the
  1980 edition, 1995.

\bibitem{XK2005}
A.Yu. Khrennikov and S.V. Kozyrev.
\newblock Wavelets on ultrametric spaces.
\newblock {\em Appl. Comput. Harmon. Anal.}, 19:61--76, 2005.

\bibitem{Kozyrev2002}
S.~V. Kozyrev.
\newblock Wavelet theory as $p$-adic spectral analysis.
\newblock {\em Izvestiya: Mathematics}, 66(2):367--376, 2002.

\bibitem{Kozyrev2004}
S.V. Kozyrev.
\newblock $p$-adic pseudodifferential operators and $p$-adic wavelets.
\newblock {\em Theoretical and Mathematical Physics}, 138(3):322--332, 2004.

\bibitem{KMR2011}
A.~Kriegl, P.W. Michor, and A.~Rainer.
\newblock {Denjoy-Carleman} differentiable perturbation of polynomials and
  unbounded operators.
\newblock {\em Integral Equations and Operator Theory}, 71(3):407--416, 2011.

\bibitem{Ledezma-Energy}
\'A.M. Ledezma.
\newblock Time-varying energy landscapes and temperature paths: Dynamic
  transition rates in locally ultrametric complex systems.
\newblock arXiv:2411.03406 [math-ph], 2024.

\bibitem{PW2024}
T.~Pierce and D.~Weisbart.
\newblock Brownian motion in the $p$-adic integers is a limit of discrete time
  random walks.
\newblock arXiv:2407.05561 [math.PR], 2024.

\bibitem{Weisbart2024}
D.~Weisbart.
\newblock $p$-adic {Brownian} motion is a scaling limit.
\newblock {\em J. Phys. A: Math. Theor.}, 57:205203, 2024.

\bibitem{Yoshida1997}
M.~Yoshida.
\newblock {\em Hypergeometric functions, my love: modular interpretations of
  configuration spaces}, volume~32 of {\em Aspects of mathematics}.
\newblock Vieweg, Braunschweig, 1997.

\bibitem{ZunigaNetworks}
W.~{Z\'{u}\~{n}iga-Galindo}.
\newblock Reaction-diffusion equations on complex networks and {Turing}
  patterns, via $p$-adic analysis.
\newblock {\em Journal of Mathematical Analysis and Applications},
  491(1):124239, 2020.

\bibitem{Zuniga2022}
W.A. Zúñiga-Galindo.
\newblock Ultrametric diffusion, rugged energy landscapes and transition
  networks.
\newblock {\em Physica A: Statistical Mechanics and its Applications},
  597:127221, 2022.

\end{thebibliography}

\end{document}